\documentclass[12pt, a4paper]{article}
\usepackage{blindtext}
\usepackage[a4paper, total={6in, 8.5in}]{geometry}
\usepackage{graphicx}
\usepackage{amsmath}
\usepackage{mathtools}
\usepackage{amssymb}
\usepackage{amsfonts}
\usepackage{dsfont}
\usepackage{stackengine} 
\usepackage[english]{babel}
\usepackage{amsthm}
\usepackage{caption}
\usepackage{subcaption}
\usepackage{xcolor}
\usepackage{soul}
\usepackage{csquotes}
\usepackage[T1]{fontenc}
\usepackage{tikz-cd}
\usepackage{adjustbox}
\usepackage{epsfig}
\usepackage{placeins}

\usepackage{hyperref}

\def\b{\mathbf}
\def\bb{\mathbb}
\def\cu{\mathcal}
\newtheorem{theorem}{Theorem}[section]
\newtheorem{corollary}{Corollary}[section]
\newtheorem{lemma}{Lemma}[section]
\newtheorem{prop}{Proposition}[section]
\newtheorem{remark}{Remark}[section]

\newcommand{\BB}{Bienaym\'e }
\newcommand{\Z}{\mathbb{Z}}
\newcommand{\R}{{\mathbb{R}}}
\newcommand{\N}{{\mathbb{N}}}
\newcommand{\E}{{\mathbb{E}}}
\newcommand{\pr}{{\mathbb{P}}}

\newcommand{\veps}{\varepsilon}
\title{\vspace{-0.8cm}\bf \Large The skeleton-blocks decomposition of \BB trees, and applications to their local convergence}
\author{\normalsize Marc A. Bernard\thanks{School of Mathematical and Physical Sciences, University of Sheffield, \url{MABernard1@sheffield.ac.uk}} \quad \& \hspace{0.2cm} Robin Stephenson\thanks{School of Mathematical and Physical Sciences, University of Sheffield, \url{robin.stephenson@normalesup.org}}}
\date{}

\begin{document}

\maketitle
\vspace{-1cm}
\begin{abstract} We study multi-type critical \BB trees through the lens of the skeleton-blocks structure, seeing them as a monotype skeleton on which we have added multi-type blocks. This decomposition enjoys a remarkably elegant compatibility with the concept of local convergence, allowing us to find simple proofs of traditional local limit theorems.
\end{abstract}

\section{Introduction}
Trees related to multi-type branching processes have received significant attention in recent literature, both on the continuous side~\cite{addarioberry2025scalinglimitsmultitypebienayme,Bertoin_Curien_Riera_2027} and the discrete side~\cite{stufler2022rerooting,tiltings2025}. One of the main questions on the discrete side is finding local limits of \BB trees. Kesten showed in \cite{Kesten} that a critical monotype \BB tree, when conditioned on having many vertices, converges locally in distribution to an infinite tree that we call the associated Kesten tree, and much of recent work has consisted in generalising this to the multi-type case, conditioning trees to be large in various ways.

In this paper, we tackle this topic by using what we call the \emph{skeleton-blocks decomposition} of \BB trees. The skeleton of a multi-type tree $\b{t}$ is the monotype tree $S(\b{t})$ obtained by removing all the non-root type vertices of $\b{t}$, but keeping the ancestral relations between the remaining elements. The blocks are then the parts which are ``between'' elements of the skeleton, with one block being attached to each vertex in the skeleton, see Figure \ref{fig:mainfig} for an example. Note that these ideas are not new: the skeleton already featured in \cite{M08} as the ``monotype projection'', and our blocks appear as the ``fringe subtrees'' of \cite{stufler2022rerooting}, while also being closely related to the ``blobs'' of \cite{addarioberry2025scalinglimitsmultitypebienayme}. However, when put together, these two concepts end up being very helpful to talk about local convergence of trees, and particularly that of \BB trees towards \textit{Kesten trees}. Kesten trees are defined in Section \ref{sec:defKesten}, and we just recall here that any critical \BB tree T, whether monotype or multi-type, has an associated Kesten tree written $\widehat{T},$ which serves as its ``infinite version''. Our main theorems is obtained directly once the theory is established:

\begin{theorem} \label{skeleton result}
Let $T$ be a non-singular, irreducible, and critical multi-type \BB tree. For $n\in\N$, let $A_n$ be an event that is measurable with respect to $\cu{S}(T)$ and let $T_n$ be distributed as $T$, conditioned on $A_n$. If $\cu{S}(T_n)\; \xrightarrow[n\rightarrow\infty]{\quad (d) \quad} \; \widehat{\cu{S}(T)}$, then \[T_n \; \xrightarrow[n\rightarrow\infty]{\quad (d) \quad} \; \widehat{T}.\]
\end{theorem}
In other words: if we condition the multi-type tree only on its  skeleton, but in a manner such that the skeleton converges to its Kesten tree, then the whole tree converges to its multi-type Kesten tree. This is relevant because we have a wealth of results giving us conditions under which a monotype tree converges to its associated Kesten tree~\cite{AD14}. In particular, taking for $A_n$ the event where $T$ has $n$ individuals of root type (which is the same as asking the skeleton to have $n$ vertices total), Theorem \ref{skeleton result} gives us one of the central results of \cite{Robin}: conditioned on having many root-type vertices, the multi-type tree converges to its associated Kesten tree. The remarkable point is that the proof requires none of the technical arguments used in \cite{Robin}.

\begin{figure}
    \centering
    \caption{The skeleton-blocks decomposition of a 3-type tree}
    \begin{subfigure}{\linewidth}
        \includegraphics[width=\linewidth]{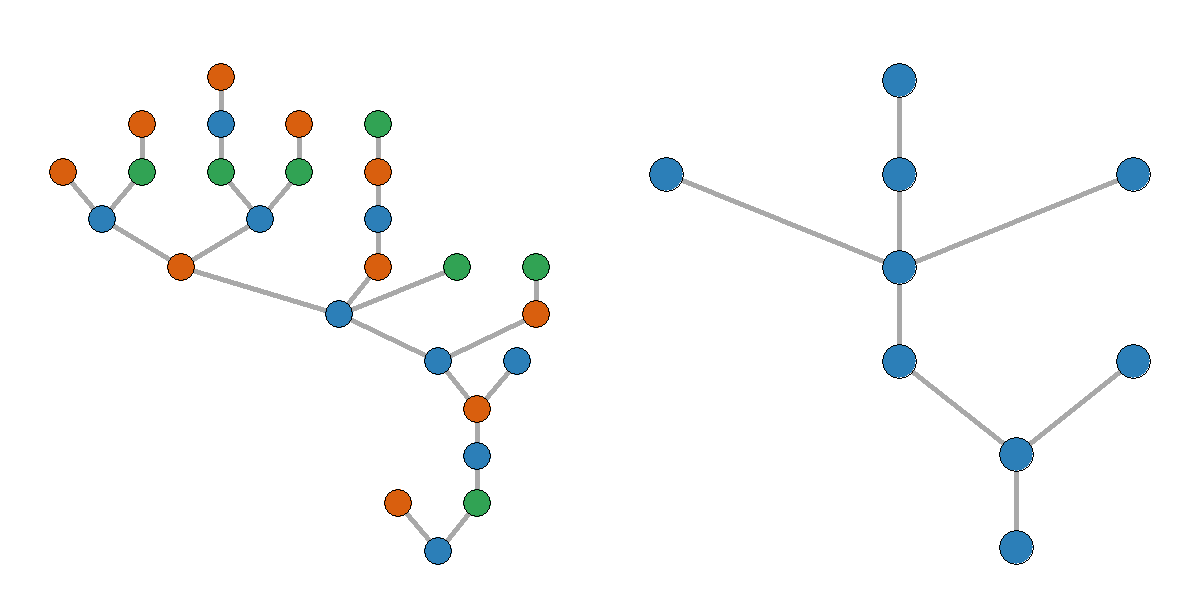}
    \caption{A 3-type tree and its skeleton.}\label{fig:mainfiga}
    \end{subfigure}
    \begin{subfigure}{\linewidth}\centering
        \includegraphics[width=0.75\linewidth]{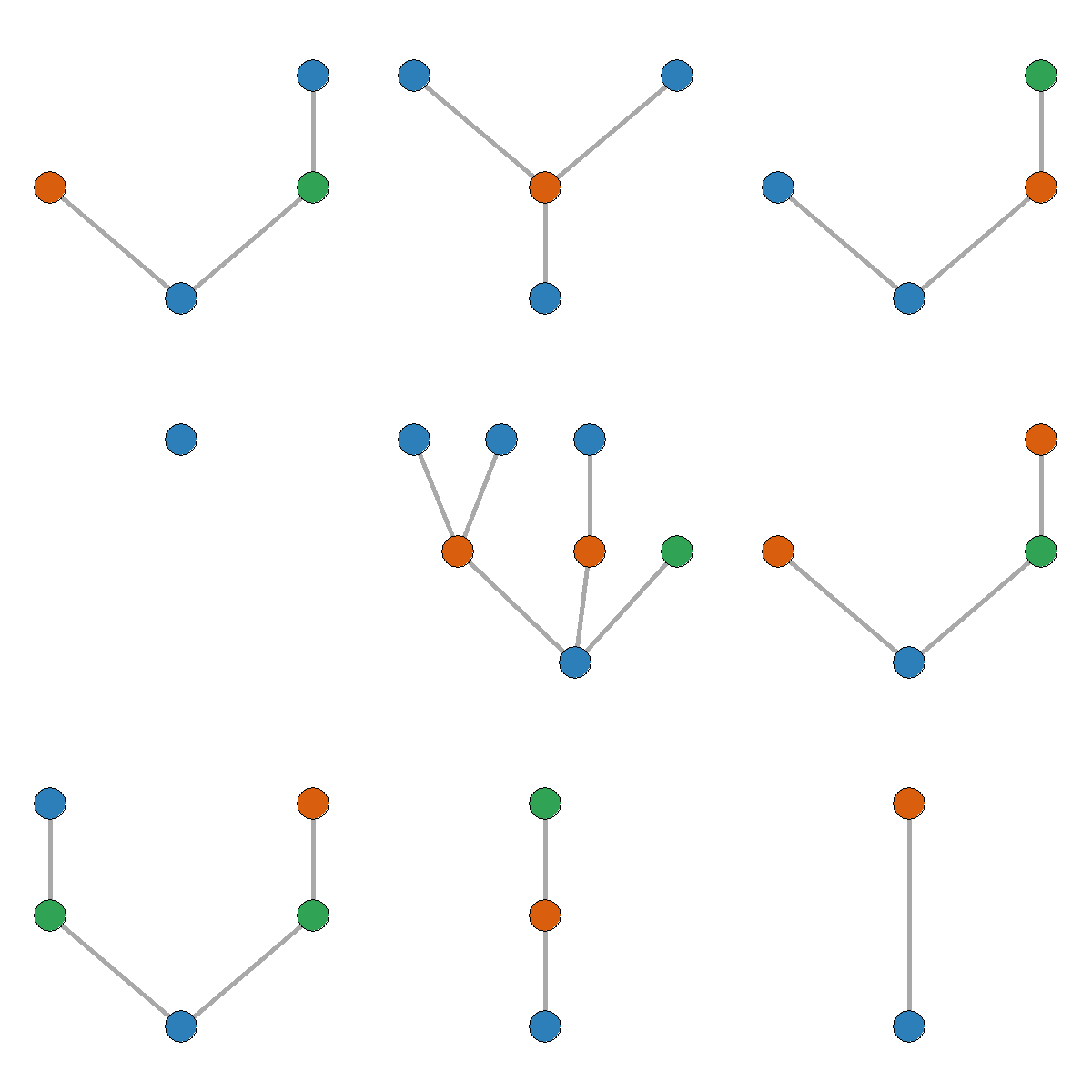}
        \caption{The blocks for the above tree, in breadth-first order.}\label{fig:mainfigb}
    \end{subfigure}
    
\end{figure}\label{fig:mainfig}
\FloatBarrier

\section{Background} \label{Background}

In this section, we set up some notation and give the necessary background on random discrete trees.
\subsection{Tree formalism and notation}
\subsubsection{Discrete Trees} \label{Discrete Trees}

We recall from Neveu's  formalism for discrete plane trees \cite{Neveu}, which are seen as subsets of the \emph{Ulam-Harris tree}
	\[\mathcal{U}=\bigcup_{k=0}^{\infty} \N^{k}.
\]
Elements of $\mathcal{U}$ are often referred to as individuals, nodes, or vertices, and usually written $u=u_1u_2\ldots u_k$ for some $k\in\Z_+$ which is the \emph{height} of $u$, and that we notate $|u|$. If $k\neq 0$ (i.e. if $u\neq \emptyset$), the \emph{parent} of $u$ is $u^-=u_1u_2\ldots u_{k-1}.$ If also  $v=v_1\ldots v_{\ell}\in\mathcal{U}$, we let $uv=u_1u_2\ldots u_kv_1\ldots v_{\ell}$ be the concatenation of $u$ and $v$. From the definition of parent, we extrapolate natural definitions for children, descendants, ancestors of individuals. 

A tree is a subset $t$ of $\mathcal{U}$ such that:
\begin{itemize}
\item $\emptyset \in {t}$, serving as root,
\item $\forall u \in {t}\setminus\{\emptyset\}, u^-\in {t}$,
\item $\forall u\in{t}, \exists k_u\in\Z_+, \forall j\in\N, uj\in{t} \Leftrightarrow j \leq k_u$ ($k_u$ can be seen as the number of children of $u$, or its degree).
\end{itemize}

We write $\bb{T}$ for the set of trees.

\subsubsection{Multi-type trees}
A $K$-type tree, or simply $K$-tree, is a pair $\b{t}=(t,e)$ where $t$ is a tree and $e$ is a function: $t\to [K]:=\{1,\ldots,K\}$, which gives a type $e(u)$ to every vertex $u$. For a vertex $u\in\mathbf{t}$, we also let $\mathbf{w}(u)=\big(e(u1),\ldots,e(uk_u)\big)$ be the list of types of the ordered offspring of $u$, which we call its \emph{$K$-degree}. Note that this is an element of $\cu{W}_K=\bigcup_{n=0}^\infty[K]^n$, the set of all finite ordered type-lists.  For $\b{w}\in\cu{W}_K$ and $i\in [K]$, we let $d_i(\b{w})$ be the number of times that $i$ features in $\b{w}$.

With a slight abuse of notation, we say that elements of $t$ are also elements of $\b{t}$. Let $\bb{T}_K$ be the set of $K$-type trees.

For $n\in\N$ and $i\in[K],$ we call the $n$-th generation of type $i$ of $\b{t},$ written $G_n^{(i)}(\b{t}),$ its set of type $i$ vertices which themselves have exactly $n$ strict ancestors of type $i$, and write $z_n^{(i)}(\b{t})$ for its cardinality.

\subsubsection{Local convergence of trees}
For $h\in\Z_+$ and $\mathbf{t}\in\bb{T}_K,$ we call $r_h(\b{t})$ the $K$-tree formed by all the individuals of $\b t$ which have at most $h$ strict ancestors with the same type as the root. The distance between two trees $\b{t}$ and $\b{t}'$ is then defined as 
\[\delta(\b{t},\b{t}')=2^{-\sup\{h\in\bb{Z}_+; \ r_h(\b{t})=r_h(\b{t}')\}},\]
where we take the convention that the supremum of $\emptyset$ is $-1.$

In other words, a ball of radius $\veps>0$ is a set of trees which are all identical up until $-\log_2(\veps)$ generations of root-type vertices. Note that this is an ultra-metric, and so all balls are both closed and open.

Having a metric gives us a Borel sigma-field, which is the one generated by the functions $r_h$ for $h\in\Z_+,$ allowing us to define random trees. Because of the discrete structure of the metric, convergence in distribution is easily characterised: a sequence of random $K$-trees $T_n$ converges in distribution to $T$ if and only if, for all $h\in\Z_+$ and all $\b{t}\in\bb{T}_K$,
\begin{equation}
\lim_{n\to\infty}\bb{P}\left[r_h(T_n)=\b{t}\right]=\bb{P}\left[r_h(T)=\b{t}\right].
\end{equation}

\begin{remark} Note that we use a slightly different definition of the local metric than the standard: traditionally, $r_h(\b{t})$ is taken to be the subset of vertices with height at most $h$. Our choice is motivated by ease of use later on, and it should be noted that the resulting topologies are very similar: while our version is slightly sharper than the usual, the only differences in convergence occur for trees which have an infinite $r_h(\b{t})$ for some $h$, which does not happen for the trees (deterministic and random) we will look at.
\end{remark}

\subsection{\BB and Kesten trees}
\subsubsection{$K$-type \BB trees}\label{sec:defBB}

An (ordered $K$-type) \emph{offspring distribution} is any family $\zeta=(\zeta^{(i)})_{i\in[K]}$, where, for all $i\in[K]$, $\zeta^{(i)}$ is a probability distribution on $\cu{W}_K$. A $K$-type \BB tree with offspring distribution $\zeta$ is then the family tree of a population whose individuals produce offspring according to $\zeta$, independently from other individuals of the same generation. In terms of notation, for $i\in[K]$, we write $\pr^{(i)}$ for a probability distribution under which $T$ is a $K$-type \BB tree with ordered offspring distribution $\zeta$, with root of type $i$. It can be characterised by the following: for any fixed $K$-tree $\b{t}$ with at most $n$ generations of type $i$,

\[\bb{P}_\zeta^{(i)}\left[r_n(T)=\b{t}\right]=\mathds{1}_{\{e(\varnothing)=i\}}\prod_{u\in\b{t}\setminus{G_n^{(i)}(\b{t})}}\zeta^{\left(e(u)\right)}\left(\b{w}(u)\right).\]

The \emph{mean matrix} $M=(m_{i,j})_{i,j\in[K]}$ is the $K\times K$ matrix whose $(i,j)$-th entry is the expected number of children of type $j$ among the offspring of an individual of type $i$: \begin{equation}\label{mean matrix}
m_{i,j}=\sum_{\b{w}\in\cu{W}_K}d_j(\b{w})\;\zeta^{(i)}(\b{w}), \qquad \forall i,j\in [K].
\end{equation}

We will usually assume that $\zeta$ is:
\begin{itemize}
    \item \textit{non-singular,} meaning that there is at least one $i$ for which $\zeta^{(i)}\big(\{\b{w}\in\mathcal{W}_K,|\b{w}|\geq 2\}\big)>0,$ avoiding purely linear trees;
    \item \emph{irreducible}, meaning that for all $(i,j)$, there exists $n$ such that the $(i,j)$-th entry of $M^n$ is positive.
\end{itemize}

Under these assumptions, if moreover the entries of $M$ are all finite, by the Perron-Frobenius theorem, the spectral radius $\rho$ of $M$ is itself an eigenvalue with multiplicity 1. We call $\boldsymbol{\alpha}$ and $\boldsymbol{\beta}$ the unique left and right eigenvectors, normalised such that  $\sum_i \alpha_i=\sum_i \alpha_i\beta_i=1$. The vector $\boldsymbol{\beta}$ should be thought of as giving a ``weight'' to vertices corresponding to their types, and this is why we will write $\beta_u$ for $\beta_{e(u)}$ when $u\in T$. Furthering this idea, we define the weight of a type-list $\b{w}\in\mathcal{W}_K$ by $\b{z}_{\b{w}}=\sum_{i=1}^{|\b{w}|}\beta_{w_i}=d(\b{w})\cdot\boldsymbol{\beta}$, where $d(\b{w})=(d_1(\b{w}),\ldots,d_K(\b{w})).$

The tree is then said to be \emph{critical} (resp. \emph{subcritical, supercritical}) if $\rho=1$ (resp. $\rho<1$, $\rho>1$).


\subsubsection{$K$-type Kesten trees} \label{sec:defKesten}
Given an ordered offspring distribution $\zeta$ as above, we can associate to it another random tree that we call the \emph{Kesten tree}. This can be seen as the family tree of a $K$-type population where individuals are further distinguished by being either \emph{normal} or \emph{special}. Normal individuals reproduce using $\zeta$, but special individuals have offspring following the size-biased distribution $\widehat{\zeta}$ given by $\widehat{\zeta}^{(i)}(\b{w})=\frac{1}{\rho}\frac{\b{z}_\b{w}}{\beta_i}\zeta^{(i)}(\b{w})$, for $i\in[K]$. Furthermore, exactly one of the children of each special individual $u$ is chosen to be special - conditionally on the offspring being $\b{w}=(w_1,\ldots,w_k)$, the special child is the $i$-th child with probability $\beta_{ui}/\b{z}_{\b{w}}.$

This gives an almost surely infinite tree with a particular infinite line of descent, often called the \emph{spine,} comprised of its special vertices. We call this tree $\widehat{T}$, still using $\pr^{(i)}$ to refer to the distribution where its root is of type $i$. If $\zeta$ is critical, then the distribution of $\widehat{T}$ is usefully characterised by the following proposition.

\begin{prop}\label{prop: K dist}
    Assume that $T$ is critical. For $\b{t}$ a $K$-tree with root of type $i\in[K]$ and  with $n\in\Z_+$ generations of type $i$, $\ell\in G_n^{(i)}(\b{t})$, the Kesten tree $\widehat{T}$ satisfies the following:

\begin{equation}\label{eq:Kestenrestrictionwithspecial}
\bb{P}^{(i)}\left[r_n(\widehat{T})=\b{t}, \ell \text{ is special}\right]=\,\bb{P}^{(i)}\left[r_n(T)=\b{t}\right].\end{equation}

\begin{equation}\label{eq:Kestenrestriction}
\bb{P}^{(i)}\left[r_n(\widehat{T})=\b{t}\right]=z_n^{(i)}(\b{t})\,\bb{P}^{(i)}\left[r_n(T)=\b{t}\right].\end{equation}

\end{prop}
\begin{proof}

We prove \eqref{eq:Kestenrestrictionwithspecial}. Note that \eqref{eq:Kestenrestriction} then follows by summing over $\ell$ in   $G_n^{(i)}(\b{t})$.

Recall the description of $\widehat{T}$ at the start of this section. By independence of individuals within a generation, we can write the probability in \eqref{eq:Kestenrestrictionwithspecial} as the product of the probabilities of each node having the correct $K$-degree, and the probabilities of choosing the correct special child of each special node. Writing $(u_0=\emptyset,u_1,\ldots,u_n=\ell)$ for the ancestral line of $\ell$ in $\b{t},$ we then have

\begin{align*}
    \pr^{(i)}\left[r_n(\widehat{T})=\b{t}, \ell \text{ is special}\right]&=\prod_{\substack{u\in\b{t}\setminus G_n^{(i)}(\b{t}) \\ u\neq u_0,\ldots,u_{n-1}}}\zeta^{\left(e(u)\right)}\left(\b{w}(u)\right)\prod_{k=0}^{n-1}\;\widehat{\zeta}^{\left(e(u)\right)}\left(\b{w}(u)\right)\frac{\beta_{u_{k+1}}}{z_{\b{w}(u_k)}\;} \\
    &=\prod_{\substack{u\in\b{t}\setminus G_n^{(i)}(\b{t}) \\ u\neq u_0,\ldots,u_{n-1}}}\zeta^{\left(e(u)\right)}\left(\b{w}(u)\right)\prod_{k=0}^{n-1}\;\frac{z_{\b{w}(u_k)}}{\beta_{u_k}}\zeta^{\left(e(u)\right)}\left(\b{w}(u)\right)\frac{\beta_{u_{k+1}}}{z_{\b{w}(u_k)}\;} \\
    &=\prod_{\substack{u\in\b{t}\setminus G_n^{(i)}(\b{t}) \\ u\neq u_0,\ldots,u_{n-1}}}\zeta^{\left(e(u)\right)}\left(\b{w}(u)\right)\prod_{k=0}^{n-1}\;\frac{\beta_{u_{k+1}}}{\beta_{u_k}}\zeta^{\left(e(u)\right)}\left(\b{w}(u)\right)\; \\
    &=\frac{\beta_{u_n}}{\beta_{u_0}} \prod_{\substack{u\in\b{t}\setminus G_n^{(i)}(\b{t}) \\ u\neq u_0,\ldots,u_{n-1}}}\zeta^{\left(e(u)\right)}\left(\b{w}(u)\right)\prod_{k=0}^{n-1}\;\zeta^{\left(e(u)\right)}\left(\b{w}(u)\right) \\
    &=\frac{\beta_{i}}{\beta_{i}}  \prod_{\substack{u\in\b{t}\setminus G_n^{(i)}(\b{t})}}\zeta^{\left(e(u)\right)}\left(\b{w}(u)\right)\\
    &= \bb{P}^{(i)}\left[r_n(T)=\b{t}\right].
\end{align*}

\end{proof}

While our interest in $\widehat{T}$ is its appearance in limits, it is worth noting that its distribution is directly obtained from that of $T$ via algebraic means.

\section{Skeletons and blocks: definitions and general properties} \label{skeletons blocks}
\subsection{Skeleton}
Given a $K$-type tree $\bf{t}$, we call its \emph{skeleton} the monotype tree obtained by removing all vertices which are not of the same type as the root, and keeping the ancestral relations. We give a formal definition taken from \cite{addarioberry2025scalinglimitsmultitypebienayme} (where it is called the ``reduced tree''). First, for all $u\in\b{t}$ of root type, let $A_u$ be the set of root-type vertices whose most recent root-type ancestor is $u$, and we index its elements in the lexicographical order by $c_1(u),\ldots,c_{|A_u|}(u).$ The skeleton $\cu{S}(\b{t})\subset{\mathcal{U}}$ is then defined inductively, along with a map $\pi:S(\b{t})\to \mathbf{t},$ as follows:
\begin{itemize}
    \item $\emptyset\in \cu{S}(\b{t}),$ the type of $\emptyset$ in $\cu{S}(\b{t})$ is the same as its type in $\b{t}$, and $\pi(\emptyset)=\emptyset,$
    \item for all $u\in \cu{S}(\b{t}),$ and $k\geq 1,$ $uk\in \cu{S}(\b{t})$ if and only if $1\leq k \leq |A_u|$; for all such $u$ and $k$ we set $\pi(uk)=c_k(\pi(u)).$
\end{itemize}

Notice that the map $\pi$ bijectively places the vertices of $\cu{S}(\b{t})$ as the root type vertices of $\b{t}.$
Since, for all $n\in\Z_+$, the ball centred at $\cu{S}(t)$ and of radius $n$ in $\bb{T}$ is entirely determined by $r_n(\cu{S}(\b{t})),$ we deduce that $\cu{S}$ is a continuous function from $\bb{T}_K$ to $\bb{T}$.

\subsection{Blocks}
In general, a \emph{block} is a $K$-type tree such that all non-root vertices have no other strict ancestor of root type - otherwise said, vertices of root type are either the root itself or a leaf. Root type leaves of a block are called \emph{hooks}. We will use the notation $\bf{b}$ for blocks, and, for convenience, we also let the empty set be a block. The set of blocks is written $\bb{B}_K$. Note that, on $\bb{B}_K$, the metric $\delta$ of $\bb{T}_K$ is entirely discrete: $\delta(\b{b},\b{b'})$ is equal to $2$ if the roots of the blocks $\b b$ and $\b{b}'$ have different types, $1$ if their roots have the same types but the blocks are not identical, and 0 if the blocks are identical.

For $\b{t}\in\bb{T}_K$ and $u\in\cu{S}(\b{t}),$ we let $B_u(\b{t})$, called the \emph{block of $u$}, be the set of all $v\in \mathcal{U}$ such that the most recent strict ancestor of type 1 of $\pi(u)v$ is $\pi(u)$. For convenience, if $u\not\in\cu{S}(\b{t}),$ we let $B_u(\b{t})=\emptyset$.

Notice that, for all $u\in\mathcal{U},$ $B_u(\b{t})$ is entirely determined by $r_{|u|+1}(\b{t}),$ making $B_u$ a continuous function.

Note that every member of $\cu{S}(\b{t})$ is the root of its own block, and, excluding the root, each is also a hook of another block.

\subsection{Two points of view on the structure of $K$-trees}\label{sec:pointsofview}
We can view $K$-trees in two different ways.

\paragraph{The ``decorated skeleton'' point of view.}

Each $K$-type tree can be seen as a skeleton ``base'', where at each skeleton vertex we insert a block, blocks being connected by identifying roots and hooks. This description is actually a bicontinuous bijection.

To be more precise, we define a function $\Phi$ from $\bb{T}\times \bb{B}_K^{\mathcal{U}}$ to $\bb{T}_k$ inductively as follows. Specifically, we build the tree $\b{t}=\Phi\big(s,(\b{b}_u,u\in\mathcal{U})\big)$ which has which has $s$ as a skeleton and $(\b{b}_u,u\in\mathcal{U})$ as blocks, as well as the function $\pi$ which injects $s$ into $\b{t},$ inductively:

\begin{itemize}
    \item $\emptyset\in \b t,$ $e(\emptyset)$ is the type of $\emptyset$ in $\b{b}_{\emptyset},$ and $\pi(\emptyset)=\emptyset,$
    \item for all $u\in \b t$ and $v\in \cu{S}(\b{t})$ such that $\pi(v)=u,$ then all vertices of the form $ux$ with $x\in \b{b}_v$ are in $\b t$, with the type of $ux$ being the type of $x$ in $\b{b}_v$. If $x$ is the $k$-th hook of $\b{b}_v$ in lexicographical order, then we let $\pi(vk)=ux.$
\end{itemize}

The function $\Phi$ clearly rebuilds a tree from its skeleton and blocks, but is not quite a bijection between $\bb{T}_K$ and $\bb{T}\times \bb{B}_K^{\mathcal{U}}$ for two reasons: $\b{b}_u$ does not matter for $u\not\in s,$ and the construction does not work if, for some $u$ in $s$, the number of hooks of $\b{b}_u$ does not match the degree of $u$. Hence we define
\[\mathcal{E}=\left\{\left(s,(\b{b}_u,u\in\mathcal{U})\right)\in \bb{T}\times \bb{B}_K^{\mathcal{U}}: \forall u\not\in s, \b{b}_u=\emptyset,\forall u\in s, z_1(\b{b}_u)=k_u(s). \right\}.\]

\begin{theorem}\label{th:generalstructure}
    The map $\Phi$
    is a homeomorphism from $\mathcal{E}$ to $\bb{T}_K$.
\end{theorem}

Note that $\mathcal{E}$ is equipped with the product topology, which itself is inherited from the product metric on $\bb{T}\times \bb{B}_K^{\mathcal{U}}$, defined as follows, where $(u_n,n\in \N)$ is an enumeration of $\mathcal{U}:$ \[\displaystyle d\bigl(\left(\mathbf{t},(\b{b}_u,u\in \mathcal{U})\right),\left(\mathbf{t}',(\b{b}'_u,u\in \mathcal{U})\right)\bigr)=\delta(\b{t},\b{t}')+\sum_{n\in\N} 2^{-n} \delta(\b{b}_{u_n},\b{b}'_{u_n}).\] Note that convergence for $d$ means convergence of every coordinate.

\begin{proof}
    It is clear from the construction that $\Phi$ bijectively maps a tree to its skeleton and blocks. More precisely, the knowledge of $r_n(\b{t})$ for $n\in\Z_+$ is equivalent to the knowledge of $r_n(s)$ and $(\b{b}_u,|u|\leq n),$ and this is exactly bicontinuity.
\end{proof}

This has useful consequences for probability and local convergence in distribution: a sequence of $K$-trees converges in distribution if and only if its skeleton converges in distribution, and the conditional distribution of the blocks given the skeleton converges in distribution.
\begin{corollary}\label{cor:generalconvergence}
Let $T$ and $T_n$, for $n\in\N,$ be random $K$-trees. 

\begin{itemize}
    \item  Let $\b{t}=\Phi\big(s,(\b{b}_u,u\in\mathcal{U})\big)$ a fixed $K$-tree with $h$ generations of root type .We have
    \[\pr[r_h(T)=\b{t}]=\pr[r_h(\cu{S}(T))=s]\;\pr\big[\,\forall u \in s, B_u(T)=\b{b}_u\mid r_h(\cu{S}(T))=s\big]\]
    \item The sequence $(T_n,n\in\N)$ converges in distribution to $T$ if and only if, for all $\b{t}=\Phi\big(s,(\b{b}_u,u\in\mathcal{U})\big)$, calling $h$ the number of root-type generations of $\b t$, we have 
    \[\underset{n\to\infty}\lim \pr[r_h(\cu{S}(T_n))=s] = \pr[r_h(\cu{S}(T))=s]\]
    and
    \[\underset{n\to\infty}\lim \pr[\,\forall u \in s, B_u(T_n)=\b{b}_u \mid r_h(\cu{S}(T_n))=s] = \pr[\,\forall u \in s, B_u(T)=\b{b}_u\mid r_h(\cu{S}(T))=s].\]
\end{itemize}
\end{corollary}

\paragraph{The ``pile of blocks'' point of view.}
A $K$-tree can also be obtained without a predetermined skeleton, simply by taking a set of blocks and stacking them wherever possible, connecting hooks and roots. To be more precise, let $(\b{b}_u,u\in\mathcal{U})$ be a collection of blocks, we build from it a $K$-tree $\b{t}=\Psi(\b{b}_u,u\in\mathcal{U})$. First, we build the skeleton $s$ by saying that $\emptyset\in s$, and that the degree of $u\in s,$ is equal to the number of hooks of $\b{b}_u.$

We then define $\Psi(\b{b}_u,u\in\mathcal{U})=\Phi\big(\big(s,(\b{b}_u,u\in\mathcal{U})\big)\big).$

\begin{prop}\label{prop:pilestructure}
    The map $\Psi$ is a continuous surjection from $(\bb{B}_K)^{\mathcal{U}}$ to $\bb{T}_K$.
\end{prop}

The proof is similar to earlier: just notice that $r_n(\Psi(\b{b}_u,u\in\mathcal{U}))$ is fully determined by $(\b{b}_u,|u|\leq n\big).$ Notice that this is not a bijection since blocks $\b{b}_u$ for $u\not\in s$ aren't used in this construction.

\section{Applications to \BB and Kesten trees}\label{structure}
\subsection{The S\&B structure of \BB and Kesten trees}
In this section and the next, we restrict ourselves to random trees with root of type 1, and omit the $(1)$ exponent in the notation. Note that there is no loss of generality. We take an ordered offspring distribution $\zeta$, and let $T$ be the corresponding \BB tree. The following results about its skeleton are found in \cite{M08}.

\begin{prop}\label{prop: Skel BB}
    The skeleton $\cu{S}(T)$ is a monotype \BB tree, and is critical/subcritical/supercritical if and only if $T$ is. If $\zeta$ admits finite $p$-th moments with $p\in\N$ (resp. admits some exponential moments), then so does the offspring distribution of $\cu{S}(T).$

    If $T$ is critical and has finite second moments, then the variance of the offspring distribution of $\cu{S}(T)$ is equal to
    \[\frac{\sigma^2}{\alpha_1\beta_1^2},\]
    where the number $\sigma> 0$ is defined by $\sigma^2=\sum_{i,j,k} \alpha_i\beta_jb_kQ^{(i)}_{j,k}$, with \\ $Q^{(i)}_{j,j}=\sum_{\mathbf{w}\in\cu{W}_K} \zeta^{(i)}(\mathbf{w}) d_j(\b{w})(d_j(\b{w})-1)$ and $Q^{(i)}_{j,k}=\sum_{\mathbf{w}\in\cu{W}_K} \zeta^{(i)}(\mathbf{w}) d_j(\b{w})d_k(\b{w})$ for $j\neq k$.
\end{prop}

We let $B$ be a reference random variable with the same distribution as $r_1(T)=B_{\emptyset}(T)$:
\[\pr\left[B=\bf{b}\right]=\pr\left[r_1(T)=\b{b}\right]=\underset{x\in\b{b}\setminus \{G_1(\b{t})\}} \prod\zeta^{\left(e(x)\right)}\bigl(\b{w}(x)\bigr),\]
and also let, for $p\in\Z_+$ for which this is defined, $B^p$ be a version conditioned on having $p$ hooks:
\[\pr\left[B^p=\bf{b}\right]=\pr\left[r_1(T)=\b{b}\mid z_1(T)=p\right]=\frac{\pr\left[B=\bf{b}\right]}{\pr\left[z_1(B)=p\right]}.\]

The expected size of this block is known, a proof of this can be found as part of the proof of Proposition 4, (ii) in ~\cite{M08}:
\begin{lemma}\label{lemma:sizeblock}
For $i\neq 1$,
\[\E[\#_i B]=\frac{\alpha_i}{\alpha_1}\]
\end{lemma}

Unsurprisingly, the blocks of $T$ have, conditionally on existing, the distribution of $B$. Better, we have the following proposition:

\begin{prop} \label{prop:S&B BB}
Let $\b{t}=\Phi\big(s,(\b{b}_u,u\in\mathcal{U})\big)$ be a fixed $K$-tree. For $u\in s,$ let $k_u$ be its degree in s. \begin{itemize}
    \item[(i)] We have
\begin{align*}
\bb{P}[r_n(T)=\b{t}]&=\prod_{u\in s\setminus G_n(s)}\bb{P}\bigl(B=\cu{B}_u(\b{t})\bigr)\\
&=\bb{P}[r_n(\cu{S}(T))=s]\;\prod_{u\in s\setminus G_n(s)}\bb{P}\bigl(B^{k_u}=\cu{B}_u(\b{t})\bigr).
\end{align*}

\item [(ii)] If moreover $E$ is an $\cu{S}(T)$-measurable event, then
\[\bb{P}[r_n(T)=\b{t},E]=\bb{P}[r_n(\cu{S}(T))=s,E]\;\prod_{u\in s\setminus G_n(s)}\bb{P}\bigl(B^{k_u}=\cu{B}_u(\b{t})\bigr).\]
\end{itemize}
\end{prop}

This gives us ways of understanding the construction of  the \BB tree based on the two points of view from Section \ref{sec:pointsofview}:
\begin{itemize}
    \item In the ``pile of blocks'' point of view, $T$ has the distribution of $\Psi(B_u,u\in\mathcal{U})$ where the $B_u$ are independent and have the distribution of $B$.
    \item In the ``decorated skeleton'' point of view, the distribution of the skeleton has already been described in Proposition \ref{prop: Skel BB}. Conditionally on the skeleton, the blocks of $T$ are independent, and the block of $u\in\mathcal{U}$ has distribution $B^{k_u}$, where $k_u$ is the degree of $u$ in the skeleton.
\end{itemize} 
The latter of these two is more useful for how it will generalise to Kesten trees.

\begin{proof}
Part (i) can be seen directly from the structure of \BB trees, as we recall from Section \ref{sec:defBB} that probabilities are equal to the products of the probabilities of each vertex (except $n$-th generation type 1 leaves) having the correct $K$-degree:
\[
    \bb{P}[r_n(T)=\b{t}]=\prod_{u\in\b{t}\setminus{G_n(\b{t})}}\zeta^{\left(e(u)\right)}\bigl(\b{w}(u)\bigr)
\]
We reorganise the product by grouping the vertices per block:
\begin{align*}
    \bb{P}[r_n(T)=\b{t}]&=\prod_{v\in s\setminus{G_n(s)}}\prod_{x\in \b{b}_v\setminus \{G_1(\b(b)_v)\}}\zeta^{\left(e(x)\right)}\bigl(\b{w}(u)\bigr) \\
    &=\prod_{v\in s\setminus{G_n(s)}}\pr[B=\b{b}_v] \\
    &=\prod_{v\in s\setminus{G_n(s)}}\pr\left[z_1(B)=k_v\right] \pr[B^{k_v}=\b{b}_v] \\
      &=\prod_{v\in s\setminus{G_n(s)}}\pr\left[z_1(B)=k_v\right] \prod_{v\in s\setminus{G_n(s)}}\pr[B^{k_v}=\b{b}_v] \\
    &=\bb{P}[r_n(\cu{S}(T))=s]\prod_{v\in s\setminus{G_n(s)}}\pr[B^{k_v}=\b{b}_v],
\end{align*}
ending the proof of (i).

For part (ii), first notice that, if $E$ depends on $\cu{S}(T)$ only up until a finite height, then this is part of (i), and then a measure theory argument extends this to general $E$.
\end{proof}

Assume now that $\zeta$ is non-singular, irreducible and critical. The Kesten tree $\widehat{T}$ is then also nicely described as a decorated skeleton.

\begin{prop} \label{prop:S&B Kesten}
\begin{itemize}
\item Since $\cu{S}(T)$ is a critical \BB tree, its own (monotype) Kesten tree $\widehat{\cu{S}(T)}$ is well defined. Moreover,
\[\widehat{\cu{S}(T)} \overset{(d)}=\cu{S}(\widehat{T}).\]

\item Let $\b{t}=\Phi\big(s,(\b{b}_u,u\in\mathcal{U})\big)\in\bb{T}_K$ be a finite $K$-tree. For $u\in s,$ let $k_u$ be its degree in s. We then have, for $n\in\N$,
\[\bb{P}[r_n(\widehat{T})=\b{t}]=\bb{P}[\cu{S}(r_n(\widehat{T}))=s]\prod_{u\in s\setminus G_n(s)}\bb{P}\bigl(B^{k_u}=\cu{B}_u(\b{t})\bigr).
\]
\end{itemize}
\end{prop}

This means that, while their skeletons are different, the conditional distributions of blocks given the skeleton is the \emph{same} in both $T$ and $\widehat{T}.$

\begin{proof}
    This is in fact immediate. Recall from Section \ref{sec:defBB} that $\bb{P}[r_n(\widehat{T})=\b{t}]=z_n(\b{t})\bb{P}[r_n(\widehat{T})=\b{t}].$ Applying Proposition \ref{prop:S&B BB}, we then have
\begin{align*}
    \bb{P}[r_n(\widehat{T})=\b{t}]&=z_n(\b{t})\bb{P}[r_n(\cu{S}(T))=s]\prod_{v\in s\setminus{G_n(s)}}\pr[B^{k_v}=\b{b}_v] \\
    &=z_n(s)\bb{P}[r_n(\cu{S}(T))=s]\prod_{v\in s\setminus{G_n(s)}}\pr[B^{k_v}=\b{b}_v] \\
    &=\bb{P}[r_n(\widehat{\cu{S}(T)})=s]\prod_{v\in s\setminus{G_n(s)}}\pr[B^{k_v}=\b{b}_v],
\end{align*}
simultaneously showing that the skeleton of $\widehat{T}$ has the same distribution as $\widehat{\cu{S}(T)}$, and that the conditioned distribution of the blocks is the same as that of those $T$.
\end{proof}

\subsection{Local limit when conditioning on the skeleton: proof of Theorem \ref{skeleton result}}

Recall that, for $n\in\N,$ $A_n$ denotes an $\cu{S}(T)$-measurable event, $T_n$ has the distribution of $T$ conditioned on $A_n$, and we assume that $\cu{S}(T_n)$ converges in distribution to $\cu{S}(\widehat{T}).$

We use Corollary \ref{cor:generalconvergence}. Let $h\in\bb{Z_+}$ and $\b{t}=\Phi\big(\big(s,(\b{b}_u,u\in\mathcal{U})\big)\big)\in\bb{T}_K.$ We have by assumption that 
$\pr[r_h(\cu{S}(T_n))=s]$ converges to $\pr[r_h(\cu{S}(\widehat{T}))=s],$ while, using part (ii) of Proposition \ref{prop:S&B BB}, and since $A_n$ is $\cu{S}(T)$-measurable:
\begin{align*}
\pr[\,\forall u \in s, B_u(T_n)=\b{b}_u \mid r_h(\cu{S}(T_n))=s] &= \pr[\,\forall u \in s, B_u(T)=\b{b}_u \mid r_h(\cu{S}(T))=s,A_n] \\
&=\prod_{v\in s\setminus{G_h(s)}}\pr[B^{k_v}=\b{b}_v].
\end{align*}
This constant value is the same as $\pr[\,\forall u \in s, B_u(T)=\b{b}_u\mid r_h(\cu{S}(\widehat{T}))=s]$ by Proposition \ref{prop:S&B Kesten}, and the proof is complete.
\qed

\section{Non-critical Cases}\label{non-crit}

This theory is mostly interesting in the case of critical \BB trees, because there is no analogue to \eqref{eq:Kestenrestriction} when the Perron eigenvalue $\rho$ of the mean matrix is not $1$. This is because, since the offspring distribution for special vertices has a factor of $\frac1\rho$, the analogue of \eqref{eq:Kestenrestrictionwithspecial} would be 
\[\bb{P}^{(i)}\left[r_n(\widehat{T})=\b{t}, \ell \text{ is special}\right]=\,\frac1{\rho^{|\ell|}}\bb{P}^{(i)}\left[r_n(T)=\b{t}\right],\] which cannot be easily summed up anymore (unless $|\ell|$ is fixed, see Section \ref{sec:periodic}). 


However, it remains interesting to ask whether $\widehat{\cu{S}(T)}\overset{(d)}=\cu{S}(\widehat{T})$ still holds.

\subsection{No general equivalence between $\widehat{\cu{S}(T)}$ and $\cu{S}(\widehat{T})$}
Rather unsurprisingly, $\widehat{\cu{S}(T)}\overset{(d)}=\cu{S}(\widehat{T})$ is not true in general, and we provide a counterexample.

Let $T$ be a two-type tree with ordered offspring distribution $\zeta$, given by \[\zeta^{(1)}(1,2)=1/2,\quad
\zeta^{(1)}(\emptyset)=1/2,\quad
\zeta^{(2)}(1)=1/2,\quad
\zeta^{(2)}(1,1)=1/2.\]
Its mean matrix is $\begin{pmatrix}
    1/2&1/2\\3/2&0
\end{pmatrix}$, with spectral radius $\rho_\zeta=\frac14(1+\sqrt{13})>1$, and the right eigenvector satisfies $\beta_1=\frac{\beta_2}6(1+\sqrt{13})$. The size-biased version of $\zeta$ can be seen to satisfy
\[\widehat{\zeta^{(1)}}(1,2)=1,\quad
    \widehat{\zeta^{(2)}}(1)=1/3,\quad
    \widehat{\zeta^{(2)}}(1,1)=2/3.\]


Let us take the root type to be $1$. The offspring distribution of the skeleton, that we call $\xi,$ satisfies
\[\xi(\emptyset)=1/2,\quad
    \xi(1,1)=1/4,\quad     \xi(1,1,1)=1/4.\]
Note that its mean is $5/4$, and we can find its size-biased version $\widehat{\xi}$, given by $\widehat{\xi}(1,1)=2/5$ and $\widehat{\xi}(1,1,1)=3/5$.

Finally we want to understand the distribution of $S(\widehat{T}).$ This tree has both special and normal vertices, with the normal ones having offspring distribution $\xi$ (since $\widehat{T}$, outside of its spine, behaves like $T$). A special node of $S(\widehat{T})$ corresponds to a special type 1 node of $\widehat{T}$; it almost surely has one child of type 1 and one of type 2, one of which is chosen to be special, with probabilities $\frac{\beta_1}{\beta_1+\beta_2}$ and $\frac{\beta_2}{\beta_1+\beta_2}$ respectively. A bit of conditional probability based on which of these is special then yields the distribution we will call $\widehat{\zeta}^S$:

\begin{align*}
        \widehat{\zeta}^S(1,1)&=\frac{\beta_1}{\beta_1+\beta_2}\frac12+\frac{\beta_2}{\beta_1+\beta_2}\frac13\\
        &=\frac{11+\sqrt{13}}{36}\simeq0.4057
    \end{align*}
    \begin{align*}
        \widehat{\zeta}^S(1,1,1)&=\frac{\beta_1}{\beta_1+\beta_2}\frac12+\frac{\beta_2}{\beta_1+\beta_2}\frac23\\
        &=\frac{25-\sqrt{13}}{36}\simeq0.5943.
    \end{align*}

Note that these are close, but not equal, to $2/5$ and $3/5$, and thus, we have $\widehat{\cu S(T)}\overset{(d)}{\neq}\cu S(\widehat{T})$. It is also worth noting that, in $S(\widehat{T}),$ the next special vertex is not chosen uniformly from the children of the current one: it is for example always equal to the leftmost child with probability $\frac{\beta_1}{\beta_1+\beta_2}.$ This argument shows that not only do we have $\widehat{\cu S(T)}\overset{(d)}{\neq}\cu S(\widehat{T}),$ but in fact $\cu S(\widehat{T})$ \emph{is not a Kesten tree at all}: there is no monotype \BB tree $U$ such that $\cu S(\widehat{T})\overset{(d)}=\widehat{U}.$

\subsection{The case of deterministic hook heights}\label{sec:periodic}
One type of non-critical tree where the skeleton still behaves well is those for which height of the root-type leaves in each block is fixed.

\begin{prop}
    \label{deterministic block height}
Let $T$ be a $K$-type \BB tree with nonsingular, irreducible, offspring distribution $\zeta$, with finite mean matrix, and taken with root of type $1$. We assume that the hooks in all the blocks are at height $d\in\N$. We then have
    \[\widehat{\cu{S}(T)}\overset{(d)}=\cu{S}(\widehat{T}).\]

\end{prop}

In order to prove this, we need some setup.

\begin{lemma}
    The mean of the offspring distribution of the monotype tree $\cu{S}(T)$ is $\rho^d$, where $\rho$ is the spectral radius of the mean matrix of $\zeta$.
\end{lemma}

\begin{proof}
    Notice first that, as a consequence of our assumption that all hooks of blocks have height $d$, almost surely no vertex of type different from $1$ has height $d$. Indeed, if this happened with positive probability, then this would yield with positive probability a block with at least one hook of height larger than $d$. As such, we can see that the first row of the matrix $M^d$ is of the form $(a,0,\ldots,0)$ with $a\in\R$, and in fact we can show that $a=\rho^d.$ Consider the eigenvector $\boldsymbol{\beta}$ for the eigenvalue $\rho$, we know that it has positive entries and that $M\boldsymbol{\beta}=\rho \beta,$ hence $M^d\boldsymbol{\beta}=\rho^d\boldsymbol{\beta}.$ Since the first row of $M^d$ is $(a,0,\ldots,0),$ we deduce by looking at the first entry of the vectors $M\boldsymbol{\beta}$ and $\rho^d \boldsymbol{\beta}$ that $a\beta_1=\rho^d\beta_1,$ and conclude that $a=\rho^n$ since $\beta_1>0.$
\end{proof}

From this we deduce in particular the distribution of $\widehat{\cu{S}(T)}.$ 

\begin{lemma}\label{lemma:distrKofSperiodic}
    For $t$ a monotype tree with height $n\in\N$, we have
\[\pr[r_n(\widehat{\cu{S}(T)})=t]=\frac{z_n(t)}{\rho^{nd}}\pr[r_n(\cu{S}(T)=t)],\]
where $z_n(t)$ is the number of vertices of height $n$ in $t$.
\end{lemma}

This can be seen either as a classical result for monotype trees (see \cite[p.7]{AD14}), or can be proven directly in the manner of Proposition \ref{prop: K dist} in the monotype case, using that the expectation of the offspring distribution is $\rho^d.$

We now write the distribution of $\widehat{T}.$ The fact that the heights of hooks is fixed means we \emph{do} have an analogue of Proposition \ref{prop: K dist} in this situation.
\begin{lemma}\label{cut kesten non crit}
For $\b t$ a $K$-tree with root of type $1$ with $n\in\Z_+$ generations of type $1$, and with $\ell\in G_n^{(1)}(\b{t})$, we have the following:
\begin{equation}
    \bb{P}\left[r_n(\widehat T)=\b t,\;\ell\text{ is special}\right] = \frac{\bb P[r_n(T)=\b t]}{\rho^{nd}},
\end{equation}
\begin{equation}\label{eq:non-crit kest}
    \bb P[r_n(\widehat T)=\b t] = \frac{z_n^{(1)}(\b{t})}{\rho^{nd}}\bb P[r_1(T)=\b t].
\end{equation}
\end{lemma}

\begin{proof}
    This is proved in much the same way as Proposition \ref{prop: K dist}. Note that the fixed height means we always get exactly $nd$ factors of $\rho$ in the denominator.
\end{proof}

\begin{proof}[Proof of Proposition \ref{deterministic block height}]
The proof is completed by noting that the skeleton is a measurable function of the tree. More specifically, write the event $\{r_n(\cu{S}(T)=t)\},$ for a fixed monotype tree $t$ of height $n$, as $\{\underset{\mathbf{t}:\cu{S}(\mathbf{t})=t}\cup r_n(T)=\b t\}.$ We can then, by \eqref{eq:non-crit kest}, write

\begin{align*} 
\bb{P}[r_n(\cu{S}(\widehat{T}))=t]&=\sum_{\mathbf{t}:\cu{S}(\mathbf{t})=t}\frac{z_n^{(1)}(\b{t})}{\rho^{nd}} \bb{P}[r_n(T)=\b t] \\
&=\frac{z_n^{(1)}(\b{t})}{\rho^{nd}} \sum_{\mathbf{t}:\cu{S}(\mathbf{t})=t}\bb{P}[r_n(T)=\b t] \\
&=\frac{z_n(t)}{\rho^{nd}}\pr[r_n(\cu{S}(T))=t] \\
&=\pr[r_n(\widehat{\cu{S}(T)})=t].
\end{align*}

This shows that $r_n(\cu{S}(\widehat{T}))$ and $r_n(\widehat{\cu{S}(T)})$ have the same distribution for all $n$, hence $\cu{S}(\widehat{T})$ and $\widehat{\cu{S}(T)}$ also do.
\end{proof}

We end this section with an example of a tree which fits the above framework when using a root of type $1$, and does not when the root is of type $2$, and turns out to only satisfy $\widehat{\cu{S}(T)}\overset{(d)}=\cu{S}(\widehat{T})$ if using the root of type $1$. Specifically, consider the following offspring distribution over $3$ types:

\[\zeta^{(1)}(2)=1/2,\quad
\zeta^{(1)}(3,3)=1/2,\quad
\zeta^{(2)}(1)=1,\quad
\zeta^{(3)}(1)=1.\]
Its mean matrix is 
$M=\begin{pmatrix}
    0 & 1/2&1 \\ 1&0&0\\ 1&0&0 \end{pmatrix},$
and the tree is easily seen to be supercritical with $\rho=\sqrt{3/2}.$ When considering a root of type $1$, this offspring distribution clearly fits the framework above with $d=2$ and so we do have $\widehat{\cu{S}(T)}\overset{(d)}=\cu{S}(\widehat{T}).$ However, if the root is of type $2$, then the expected value of the offspring distribution on the skeleton turns out to be infinite (one can see that it has the same distribution as the number of leaves of a monotype critical binary \BB tree). In this case, $\widehat{\cu{S}(T)}$ is not even defined, unlike $\cu{S}(\widehat{T})$.

\medskip \noindent \textbf{Concluding remarks.} This is an early version of this paper. Future versions will go deeper, exploring conditionings which aren't just on a function of the skeleton, the main example being adding a random variable at each root-type vertex. This will notably yield connections with monotype trees with an additive decoration at each vertex.

\medskip \noindent \textbf{Acknowledgments.}  The authors would like to thank Jonathan Jordan for some helpful discussions, and for the code which helped produce Figure \ref{fig:mainfig}. 

\bibliographystyle{alpha}
\bibliography{bibliography}

@article{AD14,
author = {Romain Abraham and Jean-Fran{\c{c}}ois Delmas},
title = {{Local limits of conditioned {G}alton-{W}atson trees: the infinite spine case}},
volume = {19},
journal = {Electronic Journal of Probability},
number = {none},
publisher = {Institute of Mathematical Statistics and Bernoulli Society},
pages = {1 -- 19},
year = {2014},
doi = {10.1214/EJP.v19-2747},
URL = {https://doi.org/10.1214/EJP.v19-2747}
}

@book{Bertoin_Curien_Riera_2027, place={Cambridge}, series={Institute of Mathematical Statistics Monographs}, title={Self-Similar Markov Trees and Scaling Limits}, publisher={Cambridge University Press}, author={Bertoin, Jean and Curien, Nicolas and Riera, Armand}, year={2027}, collection={Institute of Mathematical Statistics Monographs}}

@article{Kesten,
author = {Kesten, Harry},
journal = {Annales de l'I.H.P. Probabilités et statistiques},
language = {eng},
number = {4},
pages = {425-487},
publisher = {Gauthier-Villars},
title = {Subdiffusive behavior of random walk on a random cluster},
url = {http://eudml.org/doc/77287},
volume = {22},
year = {1986},
}

@article{tiltings2025,
  title={Existence of critical tiltings and local limits of general size-conditioned Bienaym\'e-Galton-Watson multitype trees},
  author={Poudevigne, Rémy and Thévenin, Paul},
  journal={arXiv preprint arXiv:2503.11501},
  year={2025}
}

@article{addarioberry2025scalinglimitsmultitypebienayme,
  title={Scaling limits of multitype {B}ienaym\'e trees},
  author={Louigi Addario-Berry and Philipp Beltran and Benedikt Stufler and Paul Thévenin},
  journal={arXiv preprint arXiv:2507.23241},
  year={2025}
}

@article {Robin,
    AUTHOR = {Stephenson, Robin},
     TITLE = {Local convergence of large critical multi-type
              {G}alton-{W}atson trees and applications to random maps},
   JOURNAL = {J. Theoret. Probab.},
  FJOURNAL = {Journal of Theoretical Probability},
    VOLUME = {31},
      YEAR = {2018},
    NUMBER = {1},
     PAGES = {159--205},
      ISSN = {0894-9840},
   MRCLASS = {05C63 (05C80 05C81 60B10 60J80)},
       DOI = {10.1007/s10959-016-0707-3},
       URL = {https://doi.org/10.1007/s10959-016-0707-3},
}

@article{stufler2022rerooting,
  title={Rerooting multi-type branching trees: the infinite spine case},
  author={Stufler, Benedikt},
  journal={Journal of Theoretical Probability},
  volume={35},
  number={2},
  pages={653--684},
  year={2022},
  publisher={Springer}
}

@Article{M08,
title = {Invariance principles for spatial multitype {G}alton-{W}atson trees},
author = {Miermont, Gr\'egory},
journal = {Ann. Inst. H. Poincar\'e Probab. Statist.},
volume = {44},
number={6},
year = {2008},
pages = {1128-1161},
}

@article{Neveu,
author = {Neveu, Jacques},
journal = {Ann. Inst. H. Poincar\'e Probab. Statist.},
number = {2},
pages = {199-207},
title = {Arbres et processus de {G}alton-{W}atson},
volume = {22},
year = {1986},
}
\end{document}